\documentclass[12pt]{amsart}
\usepackage{amsmath,amssymb,amsbsy,amsfonts,amsthm,latexsym,amsopn,amstext,amsxtra,epic, euscript,amscd,indentfirst}
\usepackage{graphicx}
\usepackage{enumerate}
\usepackage{tikz, tikz-3dplot}
\usepackage{listings}
\usetikzlibrary{shapes.misc}
\tikzset{cross/.style={cross out, draw=black, fill=none, minimum size=2*(#1-\pgflinewidth), inner sep=0pt, outer sep=0pt}, cross/.default={2pt}}
\begin{document}
\renewcommand{\emptyset}{\varnothing}
\newtheorem{theorem}{Theorem}
\newtheorem{conjecture}[theorem]{Conjecture}
\newtheorem{proposition}[theorem]{Proposition}
\newtheorem{question}[theorem]{Question}
\newtheorem{lemma}[theorem]{Lemma}
\newtheorem{cor}[theorem]{Corollary}
\newtheorem{obs}[theorem]{Observation}
\newtheorem{proc}[theorem]{Procedure}
\newtheorem{defn}[theorem]{Definition}
\newtheorem{remark}[theorem]{Remark}
\newcommand{\comments}[1]{} 
\def\Z{\mathbb Z}
\def\Za{\mathbb Z^\ast}
\def\Fq{{\mathbb F}_q}
\def\R{\mathbb R}
\def\N{\mathbb N}
\def\cH{\overline{\mathcal H}}
\def\cF{\mathcal F}
\def\C{\mathbb C}
\def\P{\mathcal P}

\title[Interior Hulls and Continued Fractions]{Interior Hulls of Clean Lattice Parallelograms and Continued Fractions}
\author{Gabriel Khan}
\address{GK: Department of Mathematics, Iowa State University, Ames IA, 50011}
\email{gkhan@iastate.edu}
\thanks{G. Khan is supported in part by Simons Collaboration Grant 849022}
\author{Mizan R. Khan}
\address{MRK: Department of Mathematical Sciences, Eastern Connecticut State University, Willimantic, CT 06226}
\email{khanm@easternct.edu}
\author{Riaz R. Khan}
\address{RRK: Dhaka, Bangladesh}
\email{riaz.r.khan@gmail.com}
\author{Peng Zhao}
\address{PZ: Economic Modeling LLC, Boston, MA}
\email{zhaopeng23@gmail.com}

\date{\today}

\maketitle

\begin{abstract}
    The interior hull of a lattice polygon is the convex closure of the lattice points in the interior of the polygon. In this paper we give a concrete description of the interior hull of a clean lattice parallelogram. A clean parallelogram in $\R^2$ is a lattice parallelogram whose boundary contains no lattice points other than its vertices. Using unimodular maps we can identify a clean parallelogram with a parallelogram, $P_{a,n}$, whose vertices are $(0,0), (1,0), (a,n)$ and $(a+1,n)$, with $0<a <n$ and $\gcd(a,n)=1$. Following Stark's geometric approach to continued fractions we show that the convergents of the  continued fraction of $n/a$ (viewed as lattice points) appear in a one-to-two correspondence with the vertices of the interior hull of this parallelogram. Consequently, if the continued fraction of $n/a$ has many partial quotients, then the interior hull of the corresponding parallelogram has many vertices.
    
    A pleasing consequence of our work is that we obtain an elementary geometric interpretation of the sum of the partial quotients of the continued fraction of $n/a$. Specifically, it is the difference between the area of the clean parallelogram $P_{a,n}$ and the area of its interior hull.
    
\end{abstract}

\section{Introduction}
 A lattice polygon in $\R^2$ is a polygon all of whose vertices belong to the integer lattice $\Z^2$. A beautiful elementary result for convex lattice polygons is Scott's inequality~\cite{Sco}, which states that if a convex lattice polygon contains $i>0$ lattice points in its interior and $b$ lattice points on its boundary, then 
 \begin{equation} b \leq 2i +7.\end{equation}
 Hasse and Schicho~\cite{H-S} improved on Scott's inequality by introducing a discrete analogue of the radius of the incircle in Euclidean geometry. They called this new invariant the {\it level} and denoted it as $l$. They then proved the stronger inequality 
\begin{equation}(2l-1)b \leq 2i+9l^2-2\end{equation}
for $l\geq 1.$

We should give a brief explanation of what the level is. To do so, we first define the interior hull of a lattice polygon. 

\begin{defn}
Let $P \subseteq \R^2$ be a lattice polygon containing at least one lattice point in its interior. The interior hull of $P$, denoted by $P^{(1)}$, is the convex closure of the set of lattice points in the interior of P, that is,
$$P^{(1)} = \textrm{conv}\left(\textrm{interior}(P) \cap \Z^2\right).$$
\end{defn}

Now we can obtain a nested sequence of interior hulls 
$$P^{(1)}\supseteq P^{(2)} \supseteq P^{(3)} \supseteq \ldots, $$
where $P^{(2)}$ is the interior hull of $P^{(1)}$, $P^{(3)}$ is the interior hull of $P^{(2)}$ and so on. Modulo some technicalities, the level is essentially the positive integer $n$ such that $P^{(n+1)} = \emptyset.$ Hasse and Schicho proved their inequality by considering the nested sequence of interior hulls and then peeling off these hulls. Since this peeling process is reminiscent of peeling the layers of an onion, they gave their inequality the picturesque name ``The Onion-Skin Theorem"!

In the case of a lattice rectangle with a horizontal base, the interior hull is uninteresting. It is either a lattice rectangle, a point, or a set of lattice points arranged vertically or horizontally. Consequently we decided to study lattice parallelograms. 
Our initial goal was to find examples of lattice parallelograms where the interior hull has a large number of vertices. A natural place to start is to examine \emph{clean} lattice parallelograms. For these parallelograms, there is an interesting  one-to-two correspondence between continued fractions and the vertices of $P^{(1)}$. Using this correspondence we can construct lattice parallelograms where the interior hull has a large number of vertices; thus answering our initial question. We then derive some simple consequences of our description of the vertices of $P^{(1)}$ with the most intriguing being an elementary geometric interpretation of the sum of the partial quotients of a finite continued fraction.

Our exposition and proofs are based on Stark's geometric approach to continued fractions from Chapter 7 of his introductory textbook on number theory~\cite{S}. The clarity of his exposition (both written and pictorial) played a significant role in guiding us in our work.

\subsection{Preliminaries}

 We start by giving some definitions. A natural tool in our work is the concept of an unimodular map --- the maps that preserve the lattice $\Z^n$. 

\begin{defn} 
An \emph{affine unimodular map} is an affine map 
$${\mathcal T}: \R^n \rightarrow \R^n \textrm{ of the form } {\mathcal T}(\mathbf{x}) = M\mathbf{x} + \mathbf{u},$$
where $M \in GL_n(\Z)$, $\det(M) =\pm 1$, and $\mathbf{u}\in \Z^n$. 
\end{defn}

We typically omit the word \emph{affine} when discussing such maps. Such maps give a natural definition of equivalence for lattice polygons. 

\begin{defn} 
Two lattice polygons $P_1$ and $P_2$ are said to be \emph{unimodularly equivalent} if there is an unimodular map $\mathcal T$ such that 
$${\mathcal T}(P_1) = {\mathcal T}(P_2).$$
\end{defn}

Let $P$ be a lattice parallelogram. Since translation by a lattice point preserves the lattice structure we may assume without loss of generality that one of the vertices of $P$ is the origin and that $P$ is of the form 
$$P= \{s\mathbf{u}+ t\mathbf{v} \; : \; \mathbf{u},\mathbf{v} \in \Z^2, 0\leq s,t \leq 1\}.$$
In the study of such parallelograms it is sometimes convenient to examine the quotient group  
$(\Z\oplus \Z)/(\Z\mathbf{u}\oplus \Z\mathbf{v})$. 
A basic result for this quotient group is the following.

\begin{theorem}
Let $\mathbf{u}=(u_1,u_2),\mathbf{v}=(v_1,v_2) \in \Z^2$ be linearly independent. Then,
\begin{equation}
\#\left(\frac{\Z\oplus \Z}{\Z\mathbf{u}\oplus \Z\mathbf{v}}\right)= |u_1v_2-u_2v_1|.
\end{equation}

That is, the cardinality of the quotient group equals the area of the parallelogram spanned by $\mathbf{u},\mathbf{v}$.
\end{theorem}

Since we have mentioned area, we should remind the reader of the well known theorem of Pick on lattice polygons.

\begin{theorem}[Pick]
Let $C$ be a simple polygon in $\R^2$ with vertices in $\Z^2$. Let 
$$I = \#(\Z^2 \cap interior(C)) \textrm{ and } B= \#(\Z^2 \cap boundary(C)).$$
Then,
\begin{equation}
area(C) = I+ B/2-1.
\end{equation}
\end{theorem}

The following observation will play a role in our work.

\begin{lemma}\label{antidiag-cor}
Let $P$ be a  lattice parallelogram spanned by the lattice points $\mathbf{u}=(u_1,u_2)$ and $\mathbf{v}=(v_1,v_2)$. Let $T_1$ be the lattice triangle with vertices $\mathbf{0},\mathbf{u}$ and $\mathbf{v}$ and  
let $T_2$ be the lattice triangle with vertices $\mathbf{u},\mathbf{v}$ and $\mathbf{u}+\mathbf{v}$, that is, $T_1,T_2$ are the two triangles formed by the antidiagonal of $P$. Then the correspondence
$$ s\mathbf{u}+ t\mathbf{v} \leftrightarrow (1-s)\mathbf{u}+ (1-t)\mathbf{v},
\textrm{ with } 0 < s,t < 1,$$
gives a bijective correspondence between the points of $T_1$ and the points of $T_2$.
\end{lemma}

\section{Clean Parallelograms and a Reduction Result}

For the remainder of the paper we will work with clean parallelograms. The word \emph{clean} was introduced by Reznick in~\cite{Rez}.

\begin{defn}
A lattice parallelogram in $\R^2$ is said to be \emph{clean} if the only lattice points on its sides are the vertices. If a clean lattice parallelogram does not contain any lattice points in its interior, then we call it an \emph{empty} parallelogram.
\end{defn}

We will use $P_{a,n}$ to denote 
the clean parallelogram with vertices $(0,0),(1,0),$ $(a,n)$ and $(a+1,n)$, where $1\leq a < n$ and $\gcd(a,n) =1$. Since unimodular maps preserve $\Z^2$ they map clean parallelograms to clean parallelograms. In our next result we apply this property to show that a clean parallelogram is equivalent to some $P_{a,n}$.

\begin{theorem}\label{reduction-res}
Let $P$ be the clean lattice parallelogram spanned by the lattice points $\mathbf{u},\mathbf{v}$ with $\textrm{area}(P) =n$. Then there is an unimodular map 
$T: \R^2 \rightarrow \R^2$ such that 
$$T(P) = P_{a,n}, \textrm{ with }1\leq a <n  \textrm{ and } \gcd(a,n)=1.$$
\end{theorem}

\begin{proof}
We begin by observing that there are precisely $(n-1)$ lattice points in the interior of $P$. This follows by combining Pick's theorem with the hypotheses that $\textrm{area}(P)=n$ and that the only lattice points on the boundary of $P$ are the vertices. (Another way to obtain this is to observe that each lattice point in the interior of $P$ denotes a distinct non-zero coset  of
$(\Z\oplus\Z)/(\Z\mathbf{u}\oplus\Z\mathbf{v})$ and each vertex of $P$ represents the zero coset.)

Without loss of generality we may assume that the pair of lattice points 
$$\mathbf{u}= (u_1,u_2),\, \mathbf{v}=(v_1,v_2)$$ 
are positively oriented, that is, $\det(\mathbf{u},\mathbf{v}) >0$. Since 
$\gcd(u_1,u_2) = 1$, there exist $m_1,m_2 \in \Z$ such that
$$m_1u_1+m_2u_2 =1.$$
We now consider the unimodular matrix 
\begin{equation*}
\left( \begin{array}{cc} m_1 & m_2\\ -u_2 & u_1 \end{array} \right).
\end{equation*}
This unimodular matrix maps $P$ to the clean parallelogram 
$$ P^{\prime}= \left\{ t_1 (1,0) + t_2(m_1v_1+m_2v_2,n) \; : \, 0\leq t_1,t_2 \leq 1 \right\}.$$
If $0 < (m_1v_1+m_2v_2) <n$, then 
$a=  m_1v_1+m_2v_2 .$
If $ (m_1v_1+m_2v_2) $ does not satisfy the above inequality, then we find $k\in \Z$ such that 
$0 \leq ( m_1v_1+m_2v_2  +kn) <n$  and act on $P^\prime$  by the unimodular matrix 
\begin{equation*}
\left( \begin{array}{cc}1 & k\\0& 1\end{array} \right)
\end{equation*}
to obtain $P_{a,n}$.
\end{proof}

\begin{lemma}
Let $a,n \in \N$ with $n>1$ and $\gcd(a,n)=1$. Furthermore, let $a^{-1}\in \N$, with $a <n$, such that $aa^{-1} \equiv 1 \pmod{n}.$ Then the clean parallelograms 
$P_{a,n},P_{n-a,n},P_{a^{-1},n},P_{n-a^{-1},n}$ are unimodularly equivalent.
\end{lemma}

\begin{proof}
The unimodular map 
\begin{equation*}
\mathbf{x} \mapsto \left( \begin{array}{cc} -1 & 1\\ 0 & 1 \end{array} \right)\mathbf{x} +\left(\begin{array}{c} 1\\0 \end{array}\right)
\end{equation*}
shows that $P_{n-a,n}$ is unimodularly equivalent to $P_{a,n}$.

By the extended Euclidean algorithm there exist $m \in \Z$
such that $a^{-1}a+mn =1.$ The unimodular map
\begin{equation*}
\mathbf{x} \mapsto \left( \begin{array}{cc} a^{-1} & m\\ n & -a \end{array} \right)\mathbf{x}
\end{equation*}
shows that $P_{a^{-1},n}$ is unimodularly equivalent to $P_{a,n}$.
\end{proof}

A nice aspect of working with $P_{a,n}$ is that we can easily generate the lattice points in its interior by observing that the lattice point $(1,1)$ is a generator of the cyclic group 
$\left(\Z\oplus \Z\right)/\left(\Z(1,0)\oplus \Z(a,n) \right)$. Specifically we have the following result.

\begin{lemma}The lattice points in the interior of $P_{a,n}$ are of the form 
$$ \left\langle \frac{k(n-a)}{n} \right\rangle (1,0) + \frac{k}{n}(a,n), \, k =1,\dots, n-1.$$
\end{lemma}

An elaboration of Lemma~\ref{antidiag-cor} gives an instructive aspect of the parallelograms $P_{a,n}$.

\begin{lemma}\label{isometry} The unimodular map ${\mathcal S}: P_{a,n} \rightarrow P_{a,n}$ via 
\begin{equation*}
{\mathcal S}((x,y)) =\left( \begin{array}{cc} -1 & 0\\ 0 & -1 \end{array} \right)\left( \begin{array}{c} x \\ y \end{array} \right) + \left( \begin{array}{c} a+1 \\ n \end{array} \right).
\end{equation*}
is an isometry.
\end{lemma}

At this juncture we make some simple observations about the vertices of $P_{a,n}^{(1)}$. Let $l_k$ denote the vertical line $x=k$. For $k=1,\ldots, a$, let 
$$S_k =  l_k\cap interior(P_{a,n}) \cap \Z^2.$$
(We note that the set of lattice points in the interior of $P_{a,n}$ is the union $\cup_{k=1}^a S_k$.) 
Let $L(S_k)=$ the lowest point in $S_k$, and $H(S_k)=$ the highest point in $S_k$. By noting that any point lying between 
$L(S_k)$ and $H(S_k)$ cannot be a vertex of $P_{a,n}^{(1)}$. we conclude that 
\begin{equation}
    P_{a,n}^{(1)}  = conv(\{L(S_k),H(S_k) \, : \, k=1,\ldots, a\}).
\end{equation}

\begin{lemma}\label{high-low}
If $a=1$, then $L(S_1)=(1,1)$ and $H(S_1)= (1,n-1)$.

If $a>1$, then for $k=2,\ldots, a-1,$
$$ L(S_k) = \left(k, \left\lceil \frac{(k-1)n}{a}\right\rceil \right) , 
H(S_k) = \left(k, \left\lfloor \frac{kn}{a}\right\rfloor \right);$$
$$L(S_1) =(1,1), H(S_1) = \left( 1, \left\lfloor \frac{n}{a}\right\rfloor\right);$$
and
$$ L(S_a) = \left(a, \left\lceil \frac{(k-1)n}{a}\right\rceil \right), H(S_a) = (a,n-1).$$
\end{lemma}

\begin{lemma}\label{rel-high-low}
Let $k,l\in \Z^+$ with $k+l= a+1$. Then 
$$ L(S_k)+H(S_l) = (a+1,n),$$
that is,
$$H(S_l) = {\mathcal S}(L(S_k)) \textrm{ and } L(S_k) = {\mathcal S}(H(S_l)).$$
It follows that $H(S_l)$ is a point on the boundary of $P_{a,n}^{(1)}$ if and only if the same is true for $L(S_k)$.
\end{lemma}

 The following {\sc Maple} code draws $P_{a,n}^{(1)}$. However, the pictures it generates are not particularly enlightening.

\begin{verbatim}
    chullpict := proc(a::integer, n::integer) local P, i; 
    P := {[1, 1]}; for i from 2 to n - 1 do 
    P := P union {[a*i/n + frac((n - a)*i/n), i]}; 
    end do; convexhull(P, output = plot); end proc;
\end{verbatim}

In the appendix we give Python code that generates more informative pictures of $P_{a,n}^{(1)}$. The reader may want to jump ahead to the appendix to see a picture of $P_{11,29}^{(1)},$ Figure~$\ref{fig:P11,29}$.

\section{Stark's geometric exposition of the \\ continued fraction algorithm}

In this aside we give a synopsis of Stark's approach to continued fractions presented in chapter 7 of his textbook~\cite{S}.   It is an elaboration on Klein's geometric interpretation of the continued fraction of an irrational number; see~\cite[Chapter 4, Section 12]{D}. This geometric approach is also described 
in Arnold's textbook~\cite[Chapter 3, Section 11, pages 112--113]{A} and Irwin's expository paper~\cite{I}. Stark's treatment is by far the most detailed and expository. Irwin's approach is a streamlined version of Stark, whereas Arnold (as was his wont) is exceedingly terse.

Klein's visualization of continued fractions is as follows. Let $\alpha$ be a positive number. (Typically $\alpha$ is irrational and consequently has an infinite continued fraction expansion. However, what we are describing also applies when $\alpha$ is rational with the only difference being that the continued fraction expansion is then finite.)

Consider the line 
$$y = \alpha x.$$
Let $A/B$ be a convergent arising from the continued fraction of $\alpha$. Then the lattice point $(B,A)$ is the lattice point  closest to the line $y = \alpha x$ among all the lattice points in the first quadrant that have $x$-coordinate less than or equal to $B$; that is, if a lattice point $(B^\prime,A^\prime)$, with $A^\prime, B^\prime >0$,  is closer to $y =\alpha x$ than $(B,A)$, then $B^\prime>B$.

Let
$$\frac{A_0}{B_0}, \frac{A_1}{B_1}, \ldots,\frac{A_k}{B_k}, \ldots$$
denote the convergents of $\alpha$. A basic inequality in continued fractions is 
$$\frac{A_0}{B_0}< \frac{A_2}{B_2}< \frac{A_4}{B_4} \ldots < \alpha< \ldots <
\frac{A_5}{B_5} <\frac{A_3}{B_3} < \frac{A_1}{B_1}.$$
Consequently we can view the lattice points with even subscripts, $$(B_0,A_0), (B_2,A_2), (B_4,A_4), \ldots,$$ as lying to the \emph{right} of the line $y= \alpha x$ (or lying below the line), and the lattice points with odd subscripts,
 $$(B_1,A_1), (B_3,A_3), (B_5,A_5), \ldots,$$ as lying to the \emph{left} of the line $y=\alpha x$ (or lying above the line). This visual of left and right (or above or below) will play an important role in the proof of our main result.

With this geometric viewpoint, Stark rephrased the continued fraction algorithm in the following manner.

\smallskip

\noindent {\bf CF algorithm.}

\begin{enumerate}

\item[Step 1.] Set $\mathbf{v}_{-2}=(1,0), \mathbf{v}_{-1}=(0,1).$

\item[Step 2.] Let $\mathbf{v}_0 = \mathbf{v}_{-2} + q_0\mathbf{v}_{-1}$ where $q_0$ is the largest non-negative integer 
such that 
$\mathbf{v}_{-2} + q_0\mathbf{v}_{-1}$ lies below the line $y= \alpha x$, but $\mathbf{v}_{-2} +(q_0+1)\mathbf{v}_{-1}$ 
lies above the line $y =\alpha x$.

\item[Step 3.] In general $\mathbf{v}_i = \mathbf{v}_{i-2}+ q_i\mathbf{v}_{i-1}$ where $q_i$ is the largest integer such that 
$\mathbf{v}_{i-2}+ q_i\mathbf{v}_{i-1}$  lies on the same side of the line $y = \alpha x$  as $\mathbf{v}_{i-2}$, but 
$\mathbf{v}_{i-2}+ (q_i+1)\mathbf{v}_{i-1}$ lies on the opposite side of $y = \alpha x$.

\item[Step 4.] The algorithm terminates if $\mathbf{v}_i$ lies on the line $y =\alpha x$.

\end{enumerate}

For the rest of this section we will use the notation introduced in the above algorithm. A basic theorem on continued fractions can be rephrased as follows.

\begin{theorem}\label{converge det}
For $i=-2,-1,0,1\dots$,
\begin{equation}
\det\left(\mathbf{v}_i,\mathbf{v}_{i+1}\right) = (-1)^i.
\end{equation}
Consequently, the parallelogram spanned by $\mathbf{v}_i$ and $\mathbf{v}_{i+1}$ does not contain any non-vertex  lattice points, and 
$\Z\mathbf{v}_i \oplus \Z\mathbf{v}_{i+1} =\Z^2$.
\end{theorem}

Let $a,n \in \Z^+$ with $a<n$ and $\gcd(a,n)=1$. Let $\mathbf{v}_{-2} =(1,0),$ and $\mathbf{v}_{-1}=(0,1)$. We express the continued fraction expansion of $n/a$ as
\begin{eqnarray*}
(B_0,A_0)=\mathbf{v}_0  & = &  \mathbf{v}_{-2}+ q_0\mathbf{v}_{-1}\\
(B_1,A_1)= \mathbf{v}_1 & = & \mathbf{v}_{-1} + q_1\mathbf{v}_0\\
(B_2,A_2)= \mathbf{v}_2 & = & \mathbf{v}_{0} + q_2\mathbf{v}_1\\
& \vdots & \\
(B_{m-1},A_{m-1})= \mathbf{v}_{m-1} &= & \mathbf{v}_{m-3}+ q_{m-1}\mathbf{v}_{m-2}\\
(a,n)= \mathbf{v}_m &= & \mathbf{v}_{m-2}+ q_m\mathbf{v}_{m-1}.
\end{eqnarray*}

\begin{figure}
\begin{center}
\begin{tikzpicture}[scale=1.1]
\coordinate [label=below left:0] (0) at (0,0);

\node at (3,1) {$l$};

\draw[-] (0,0) -- (9,0);
\draw[-] (0,0) -- (0,2);
\draw[-] (0,2) -- (9,2);
\draw[-] (9,0) -- (9,2);
\draw[-] (7,0) -- (7,2);
\draw [thick] (0,0) -- (10,2.5);
\node at (10.2,2.75) {$\mathbf{v}_m =(a,n)$};

\draw[fill] (0,0) circle [radius=0.08];
\draw (2,0) node[cross]{}; \draw (0,2) node[cross]{}; \draw (7,0) node[cross]{}; \draw (7,2) node[cross]{}; \draw (9,0) node[cross]{};
\draw (9,2) node[cross]{}; \draw (10,2.5) node[cross]{};
\node at (2,-0.25) {$\mathbf{v}_{i-1}$};
\node at (0,2.25) {$\mathbf{v}_{i-2}$};
\node at (7,2.25) {$\mathbf{v}_{i}$};
\node at (7,-.25) {$q_i\mathbf{v}_{i-1}$};
\node at (9,-.25) {$(q_i+1)\mathbf{v}_{i-1}$};

\end{tikzpicture}
\end{center}
\caption{The CF algorithm}
\label{fig:CF algorithm}
\end{figure}
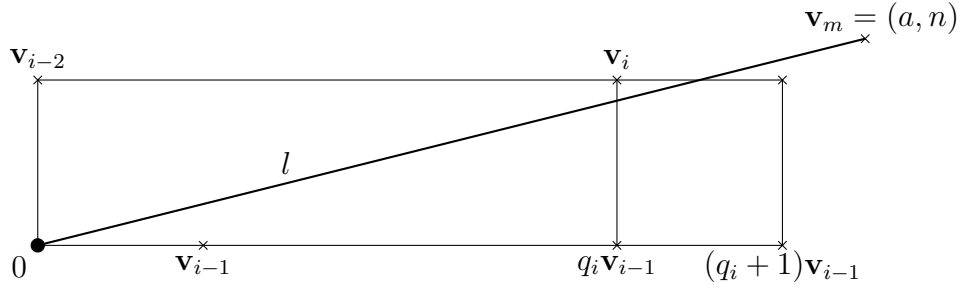

Figure~\ref{fig:CF algorithm} is a visual illustration of the continued fraction of algorithm. We warn the reader that we have taken some artistic liberties. We have chosen the vector $\mathbf{v}_{i-1}$ to be horizontal and the vector $\mathbf{v}_{i-2}$ to be vertical; in reality they both have positive slope. Furthermore we have only indicated some of the lattice points that lie on the boundary of the parallelogram with vertices $(0,0), \mathbf{v}_{i-2}, \mathbf{v}_i$ and $q_i\mathbf{v}_{i-1}$. There are no lattice points on the line segment 
connecting $(0,0)$ to $\mathbf{v}_{i-2}$ other than the endpoints. The same remark holds for the line segment connecting $q_i\mathbf{v}_{i-1}$ to $\mathbf{v}_i$. However, on the line segment connecting $(0,0)$ to 
$q_i\mathbf{v}_{i-1}$ there are $q_i-1$ lattice points in addition to the end points. The same remark holds for the line segment 
connecting $\mathbf{v}_{i-2}$ to $\mathbf{v}_i$. 

The following little lemma shows that $q_m \geq 2$, and consequently there is at least one lattice point lying on the open line segment connecting $\mathbf{v}_{m-2}$ and $\mathbf{v}_m$. We will need this observation when we are determining the vertices of $P_{a,n}.$

\begin{lemma}\label{last-par-quot}
We recall that $[q_0,q_1,q_2,\ldots,q_m]$ is the continued fraction of $n/a$. We  have the inequality  
\begin{equation}
q_m \geq 2.
\end{equation}
\end{lemma}

\begin{proof}
We apply the Euclidean algorithm with $n$ and $a$. \begin{eqnarray*}
n & = & q_0a+r_0 \\
a & = & q_1r_0 + r_1 \\
\vdots & \vdots &\vdots \\
r_{m-4} & = & q_{m-2}r_{m-3}+r_{m-2} \\
r_{m-3} & = & q_{m-1}r_{m-2}+ r_{m-1} \\
r_{m-2} & = & q_m r_{m-1}, \\
\end{eqnarray*}
with 
$$r_0 > r_1 > \ldots > r_{m-2} > r_{m-1}.$$
Since, $r_{m-1} =1 $, we conclude 
$q_m =r_{m-2} > 1.$
\end{proof}

\section{Continued fractions and the boundary of $P_{a,n}^{(1)}$}

In this section we prove our main result. We implement Stark's approach to show how the continued fraction of $n/a$ gives the vertices of $P_{a,n}^{(1)}$. 

\subsection{A menagerie of notations and preliminary results}

In this section we discuss some results about the convergents of a continued fraction. Some of them will be used in the proof of our main theorem; others are there to illuminate the geometry.

For $i=0,1,\ldots,m-1$ let 
$$ d_i = \textrm{ distance from } (B_i, A_i) \textrm{ to the line } y = \frac{n}{a} x.$$
The next result shows that we can view the distances, $d_i$, as successive ``minima". 

\begin{theorem}\label{dist-prop}
We have the strictly decreasing  sequence 
\begin{equation}
\frac{\textrm{width}(P_{a,n})}{2}= \frac{n}{2\sqrt{n^2+a^2}} > d_0 > d_1 > d_2 > \ldots > d_{m-1} > 0.
\end{equation}

Furthermore, let $(\beta,\alpha)$ be a lattice point in the interior of the rectangle with vertices $(0,0), (0,n), (a,0)$ and $(a,n)$ satisfying the following conditions:
$$0<\beta \leq B_i, \, (\beta,\alpha) \not= (B_{i-1},A_{i-1}), \, (\beta,\alpha) \not=(B_{i},A_{i}).$$
Then, $d> d_{i-1}$, where d is the distance of $(\beta,\alpha)$ to the line $y=(n/a)x.$
\end{theorem}

We now use the lines 
$y = (n/a)(x-0.5)$ and and $y=n/2$ 
 to subdivide $P_{a,n}$ into four smaller parallelograms $P_1,P_2, P_3$ and $P_4$ (see Figure~\ref{fig:line picture}). These four parallelograms provide a helpful visual partition for the vertices of $P^{(1)}_{a,n}$.  We also introduce two numerical quantities $E(m)$ and $O(m)$. Let 
\begin{equation} E(m) = \textrm{ largest even integer }< m \end{equation}
and
\begin{equation} O(m)= \textrm{ largest odd integer } < m.
\end{equation}
 
\begin{lemma}\label{various-incl}
Let $V_1,V_2$ denote the sets 
$$V_1= \{ (B_0,A_0), (B_2, A_2), (B_4, A_4), \ldots, (B_{E(m)}, A_{E(m)})\}, $$
$$V_2= \{(B_1+1, A_1), (B_3 + 1, A_3), (B_5+1,A_5), \ldots, (B_{O(m)} +1, A_{O(m)}) \} $$
respectively, and recall the unimodular map $\mathcal S$ in Lemma~\ref{isometry} given by ${\mathcal S}((x,y)) = (a+1,n)-(x,y).$
Then
$$ V_1 \subset P_1, \, {\mathcal S}(V_1) \subset P_3, \, V_2\subset P_2, \textrm{ and } {\mathcal S}(V_2) \subset P_4.$$ 
\end{lemma}

\begin{proof}
We begin by showing that the points in $V_1$ and $V_2$ lie below the horizontal line $y=n/2$ and the points in ${\mathcal S}(V_1)$
and ${\mathcal S}(V_2)$ lie above the horizontal line $y=n/2$.

We make two observations: 
\begin{enumerate}
    \item The points in $V_1$ and $V_2$ lie to the right of the line $y=(n/a)x$ and to the left of the line $y=(n/a)(x-1)$. 
\item We have the following sequence of inequalities: 
$$ 1 \leq A_0 < A_1 < \ldots < A_{m-1} < A_m =n.$$
\end{enumerate}
The inequality $q_m \geq 2$ (Lemma~\ref{last-par-quot}) gives us the inequality 
$$ n = A_{m-2} + q_mA_{m-1} > q_mA_{m-1} \geq  2A_{m-1},$$ 
and we can conclude that 
$$n/2 > A_i \textrm{ for } i=0, \ldots, m-1.$$ 
Furthermore, the $y$-coordinates of the points in 
${\mathcal S}(V_1) \cup {\mathcal S}(V_2)$ are of the form 
$n-A_i$. Consequently the points in $V_1$ and $V_2$ lie below the horizontal line  $y=n/2$ and the points in ${\mathcal S}(V_1)$ and ${\mathcal S}(V_2)$ lie above the line $y=n/2$.

We now show that the convergent $(1,A_0)$ lies in the interior of $P_1$. The point $(1,A_0)$ lies below the line $y=(n/a)x$ and does not lie on the line $y=(n/a)(x-1/2).$ If $(1,A_0)$ lay below the line $y=(n/a)(x-1/2)$, then the lattice point $(1,2A_0)$ would lie below the line $y=(n/a)x$ and would be \emph{closer} to $y=(n/a)x$ than $(1,A_0)$. This contradicts Theorem~\ref{dist-prop}. Consequently, $(1,A_0)$ must lie inside $P_1$. 

We now recall that the distances of the points in $V_1$ to the line $y = (n/a)x$ are $d_0,d_2,\ldots, d_{E(m)}$ and the distances of the points in $V_2$ to the line $y=(n/a)(x-1)$ 
are $d_1,d_3, \ldots, d_{O(m)}$. Since 
$$\textrm{width}(P_1) =\textrm{width}(P_2) > d_0 > d_1 > \ldots > d_{m-1},$$ 
we conclude that $V_1 \subset P_1$ and $V_2 \subset P_2$. Finally a routine check shows that ${\mathcal S}(P_1) = P_3$ and 
${\mathcal S}(P_2) = P_4$.
\end{proof}

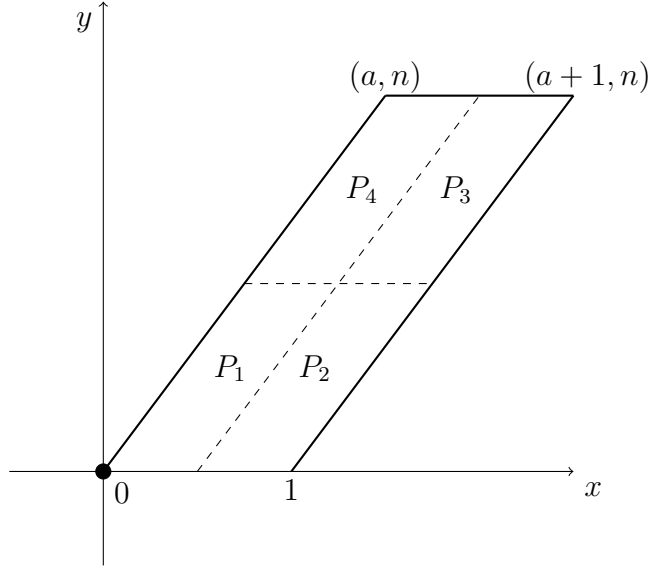
\begin{figure}

\begin{center}
\begin{tikzpicture}[scale=1.25]
\coordinate [label=below right:0] (0) at (0,0);
\coordinate [label=below right:$x$] ($x$) at (5,0);
\coordinate [label=below left:$y$] ($y$) at (0,5);

\draw[->] (-1,0) -- (5,0);
\draw[->] (0,-1) -- (0,5) ;
\draw [thick] (0,0) --(3,4);
\node at (3,4.2) {$(a,n)$};
\draw [dashed] (1,0) -- (4,4);
\draw[thick](2,0)--(5,4);
\node at (5.15,4.2){$(a+1,n)$};
\draw[thick] (3,4) -- (5,4);
\draw[dashed] (1.5,2) -- (3.5,2);
\draw[fill] (0,0) circle [radius=0.08];
\node at (1.35,1.1) {$P_1$};
\node at (2.25, 1.1) {$P_2$};
\node at (3.75,3.0) {$P_3$};
\node at (2.75,3.0) {$P_4$};
\node at (2,-.2) {1};

\end{tikzpicture}
\end{center}
\caption{Subdividing $P_{a,n}$}
\label{fig:line picture}
\end{figure}

We now prove 3 small results that we will use to prove our main theorem.
The next lemma follows immediately from examining Figure~\ref{fig:CF algorithm} and applying Theorem~\ref{converge det}. However, we apply Pick's theorem to give a proof.

\begin{lemma}\label{no-int-pt}
For $i=0,1\ldots,m$, consider the triangle, $T_i$, with vertices $(0,0), \mathbf{v}_{i-2},$ and $\mathbf{v}_{i}$. Then we have the following:
\begin{enumerate}
\item There are $q_i+1$ lattice points on the line segment connecting $\mathbf{v}_{i-2}$ to $\mathbf{v}_i$.
\item There are no lattice points in the interior of $T_i$.
\item All of the non-vertex lattice points in $T_i$  lie on the line segment connecting $\mathbf{v}_{i-2}$ to $\mathbf{v}_i$.
\end{enumerate}
\end{lemma}

\begin{proof}
The reader may want to refer to Figure~\ref{fig:CF algorithm}. We have that 
$$\mathbf{v}_i -\mathbf{v}_{i-2}= \mathbf{v}_{i-2} + q_i\mathbf{v}_{i-1}-\mathbf{v}_{i-2}= q_i(B_{i-1},A_{i-1}).$$
It follows that the number of lattice points on the line segment connecting $\mathbf{v}_{i-2}$ to $\mathbf{v}_i$ equals 
$$q_i\gcd(B_{i-1},A_{i-1}) +1 = q_i+1.$$
Furthermore, since $\mathbf{v}_{i-2}$ and $\mathbf{v}_i$ are convergents there are no non-vertex lattice points on the other two sides of $T_i$. Combining these two remarks we conclude that there are $q_i+2$ lattice points on the boundary of $T_i$.

We now have 
$$\textrm{area}(T_i) = |\det(\mathbf{v}_{i-2},\mathbf{v}_i)|/2
= |q_i\det(\mathbf{v}_{i-2},\mathbf{v}_{i-1})|/2.$$
Since $|\det(\mathbf{v}_{i-2},\mathbf{v}_{i-1})|=1$ (Theorem~\ref{converge det}), we have that 
$\textrm{area}(T_i) = q_i/2.$ We now invoke Pick's theorem to conclude that $T_i$ does not contain any interior lattice points, and consequently all of the non-vertex lattice points in $T_i$ lie on the edge connecting $\mathbf{v}_{i-2}$ to $\mathbf{v}_i$.
\end{proof}



{\bf Notation for piecewise linear paths}: We will use 
$\gamma(\mathbf{w}_1,\mathbf{w}_2,\ldots, \mathbf{w}_n)$
to denote the piecewise linear path that connects $\mathbf{w}_1$ to $\mathbf{w}_2$, $\mathbf{w}_2$ to $\mathbf{w}_3$, and so on and so forth.

\begin{lemma}\label{angle-at-conv}
For $i=0,\ldots,m-2$,
\begin{equation}
    \det \left( \begin{array}{ccc} 1 & B_{i-2} & A_{i-2}\\ 
    1 & B_i & A_i \\ 1 & B_{i+2} & A_{i+2} \end{array} \right) = 
    (-1)^{i-1}q_iq_{i+1}q_{i+2}.
\end{equation}
Consequently, when traversing the piecewise linear paths 
$$\gamma(\mathbf{v}_1, \mathbf{v}_3, \ldots, \mathbf{v}_{O(m)})
\textrm{ and } \gamma(\mathbf{v}_{E(m)},\mathbf{v}_{E(m)-2}, \ldots, \mathbf{v}_0),$$
we are always turning to the left.
\end{lemma}

\subsection{Our main theorem}
For $i=-2, \ldots, m$ let
$$\mathbf{u}_i = \mathbf{v}_i + (1,0).$$
It is worth noting that $\mathbf{u}_{-1} = (1,1)$ and $\mathbf{u}_m = (a+1,n).$
\begin{figure}
\begin{center}
\begin{tikzpicture}[scale=1]
\node at (0,-.5) {(0,0)};
\draw[fill] (0,0) circle [radius=.05];
\node at (0,10.5){$(a,n)$};
\draw[fill] (0,10) circle [radius=.05];
\node at (10,10.5){$(a+1,n)$};
\draw[fill] (10,10) circle [radius=.05];
\node at (10,-.5){(1,0)};
\draw[fill] (10,0) circle [radius=.05];
\node at (5,-.5){(1/2,0)};
\draw[fill] (5,0) circle [radius=.05];
\node at (-1.1,5){$(a/2,n/2)$};
\draw[fill] (0,5) circle [radius=.05];
\node at (11.2,5){$(a+1,n/2)$};
\draw[fill] (10,5) circle [radius=.05];
\node at (5,10.5){$(a+1/2,n)$};
\draw[fill] (5,10) circle [radius=.05];

\node at (2.5,4.1){$pt_7$};
\draw[fill] (2,4) circle (2pt);
\draw [thick] (0,0) -- (2,4);
\node at (4.3,2.3){$pt_8$};
\draw[fill] (4,2) circle (2pt);
\draw [dashed] (2,4) -- (4,2);
\draw [thick] (0,0) -- (4,2);
\draw [thick] (4,2) -- (10,0);
\node at (7,0.6){$pt_1$};
\draw [fill] (7,1) circle (2pt);
\node at (9.5,3){$pt_2$};
\draw [dashed] (7,1) -- (9,3);
\draw[fill] (9,3) circle (2pt);
\draw [thick] (10,0) -- (8,6);
\node at (0.6,7){$pt_6$};
\draw [fill] (1,7) circle (2pt);
\draw [thick] (2,4) -- (0,10);
\node at (3.1,9.3){$pt_5$};
\draw [fill] (3,9) circle (2pt);
\draw [dashed] (1,7) -- (3,9);
\draw [thick] (0,10) -- (3,9);
\node at (5.9, 7.6){$pt_4$};
\draw [fill] (6,8) circle (2pt);
\draw [thick] (3,9) -- (6,8);
\draw [thick] (8,6) -- (10,10);
\draw[fill] (8,6) circle (2pt);
\node at (7.6,5.9){$pt_3$};
\draw[thick] (6,8) -- (10,10);
\draw [dashed] (6,8) -- (8,6);
\draw [dotted] (0,5) -- (10,5);
\draw [dotted] (5,0) -- (5,10);

\node at (4.5,4.5){$R_0$};
\node at (1,4){$T_m$};
\node at (2,2){$R_1$};
\node at (4,1){$T_0$};
\node at (8.5,1.5){$R_2$};
\node at (9.25,5.5){${\mathcal S}(T_m)$};
\node at (8,8){${\mathcal S}(R_1)$};
\node at (6,9){${\mathcal S}(T_0)$};
\node at (1.5,8.5){${\mathcal S}(R_2)$};

\node at (7.75,2.25){$\gamma_1$};
\node at (2.75,2.75){$\gamma_2$};
\node at (6.5,6.75){${\mathcal S}(\gamma_2)$};
\node at (2.25,7.5){${\mathcal S}(\gamma_1)$};

\draw[thick] (0,0) -- (0,10) -- (10,10) -- (10,0) -- (0,0);

\end{tikzpicture}
\end{center}
\caption{A schematic diagram of $P_{a,n}$ with 
$R_0 =P_{a,n}^{(1)}$}
\label{fig:schematic picture}
\end{figure}
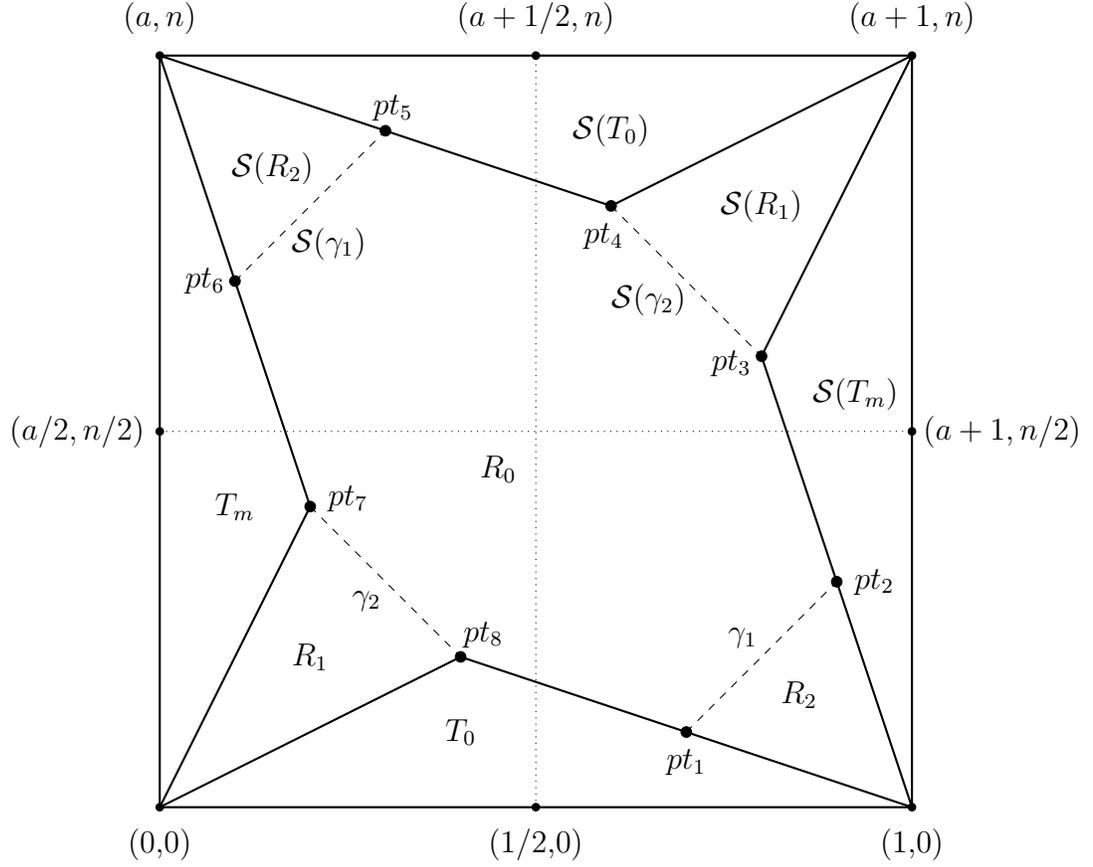

\begin{theorem}\label{main result}
The boundary of $P_{a,n}^{(1)}$ is given by the paths 
$$\gamma(\mathbf{u}_{-1}, \mathbf{u}_1, \mathbf{u}_3, \ldots, \mathbf{u}_{O(m)},\mathbf{u}_m - \mathbf{v}_{E(m)}),$$
$$\gamma(\mathbf{u}-\mathbf{v}_{E(m)},\mathbf{u}_m - \mathbf{v}_{E(m)-2}, \ldots,\mathbf{u}_m - \mathbf{v}_0,\mathbf{u}_m-\mathbf{u}_{-1}),$$ 
$$\gamma(\mathbf{u}_m - \mathbf{u}_1, \mathbf{u}_m - \mathbf{u}_3, \ldots, \mathbf{u}_m-\mathbf{u}_{O(m)},\mathbf{v}_{E(m)}),$$
and $$\gamma(\mathbf{v}_{E(m)},\mathbf{v}_{E(m)-2}, \ldots, \mathbf{v}_0, \mathbf{u}_{-1}).$$
Furthermore, these lattice points are the vertices of $P_{a,n}^{(1)}$. (We will denote the union of these 4 paths by $\Gamma$.)
\end{theorem}

Two points are worth noting. The lattice point 
$\mathbf{u}_{-1}=(1,1)$ is always a vertex of $P_{a,n}^{(1)}$. Furthermore, when 
$2a >n$, then $\mathbf{v}_0=(1,1)$ and $\mathbf{u}_m-\mathbf{v}_0= \mathbf{u}_m-\mathbf{u}_{-1}$.

\begin{proof}

We begin observing that there are two possible cases: 
$$ (i)\,\, E(m) < O(m); \,\, (ii) \,\, O(m) < E(m).$$
Since the proof for the two cases are identical we restrict ourselves to the case $E(m) < O(m)$. 

The essence of the proof is captured in Figure~\ref{fig:schematic picture}. Before proceeding with our proof we should make some clarifying remarks about this picture.

In the picture we assume that $2a < n$ (that is $\mathbf{v}_0\not=(1,1)$) and $E(m) < O(m)$, and we retain this assumption in our proof.  (The distinction between the cases 
$\mathbf{v}_0 \not= (1,1)$ and $\mathbf{v}_0 =(1,1)$ is trivial.)  We list some aspects of Figure~\ref{fig:schematic picture}. (At this juncture the reader may want to compare Figure~$\ref{fig:schematic picture}$ with Figure~$\ref{fig:P11,29}$.)

\begin{enumerate}

\item A dashed line represents a piecewise linear path. The dotted lines show us the decomposition of $P_{a,n}$ into the 4 congruent parallelograms $P_1,P_2,P_3,P_4$. We remind the reader that $\mathcal S$ represents the unimodular map 
$${\mathcal S}(\mathbf{v}) =-\mathbf{v}+(a+1,n)$$.

\item The polygonal region enclosed by $\Gamma$ is denoted $R_0$. The way we have labelled the points $pt_1, \ldots, pt_8$ in Figure~\ref{fig:schematic picture} illustrates our visualization of $\Gamma$ as starting at $pt_1= \mathbf{u}_{-1}=(1,1)$, travelling counterclockwise and eventually returning to $pt_1$. 

\item $pt_6$ is the \emph{highest} lattice point on the line segment connecting $\mathbf{v}_{E(m)}$ to $(a,n)$ that is distinct from $(a,n)$. Since 
$$ (a,n) = \mathbf{v}_{E(m)} + q_m \mathbf{v}_{O(m)},$$
we conclude that 
$$ pt_6 = \mathbf{v}_{E(m)} + (q_m-1)\mathbf{v}_{O(m)} = \mathbf{u}_m-\mathbf{u}_{O(m)}.$$

\item The $pt_i$'s are the following lattice points: 
$pt_1 = \mathbf{u}_{-1}$, $pt_2= \mathbf{u}_{O(m)}$, $pt_3 = \mathbf{u}_m - \mathbf{v}_{E(m)}$, 
$pt_4 = \mathbf{u}_m - \mathbf{v}_0$, $pt_5= \mathbf{u}_m-\mathbf{u}_{-1}$, $pt_6= \mathbf{u}_m-\mathbf{u}_{O(m)}$, 
$pt_7= \mathbf{v}_{E(m)}$,  $pt_8 = \mathbf{v}_0$. 

\item $\gamma_1$ is the piecewise linear path 
$\gamma(\mathbf{u}_{-1}, \mathbf{u}_1, \ldots, \mathbf{u}_{O(m)})$ and $\gamma_2$ is the piecewise linear path 
$\gamma(\mathbf{v}_0, \mathbf{v}_2, \ldots, \mathbf{v}_{E(m)})$. 

\item We recall the notation from Lemma~\ref{no-int-pt} where 
$T_i$ denoted the triangle with vertices $(0,0)$, $\mathbf{v}_{i-2}$ and $\mathbf{v}_i$. The regions $R_1,R_2$ arise from 
unions of certain $T_i$'s. Specifically,
$$ R_1 = \cup_{i=2, i \textrm{ even }}^{E(m)} T_i 
\textrm{ and } R_2 = \left( \cup_{i =1, i\textrm{ odd }}^{O(m)} T_i \right) + (1,0).$$

\end{enumerate}

Our goal is to prove that 
\begin{equation}
    R_0 = P_{a,n}^{(1)}.
\end{equation}
The proof consists of proving two things:
\begin{enumerate}
    \item All of the lattice points in the interior of $P_{a,n}$ are in $R_0$, that is,
    $$(P_{a,n} - R_0) \cap \Z^2 = \{(0,0),(1,0),(a+1,n),(a,n)\}.$$
    
    \item $R_0$ is convex.
\end{enumerate}

In Lemma~\ref{no-int-pt} we proved that the triangles $T_i$ did not contain any interior lattice points. Consequently, there are no interior lattice points in any of the regions 
$ T_m, R_1, T_0$ and $R_2$.
Furthermore, since $\mathcal S$ is an unimodular map the same remark holds for the regions 
${\mathcal S}(T_m)$, ${\mathcal S}(R_1)$, ${\mathcal S}(T_0)$ and 
${\mathcal S}(R_2)$. Thus to complete the proof that   
$$(P_{a,n} - R_0) \cap \Z^2 = \{(0,0),(1,0),(a+1,n),(a,n)\}$$
we need to show that the only lattice points on the line segments connecting $(a,m)$ to $pt_6= \mathbf{u}_m-\mathbf{u}_{O(m)}$, $(0,0)$ to $pt_7= \mathbf{v}_{E(m)}$, $(0,0)$ to $pt_8 =\mathbf{v}_0$ and $\mathbf{u}_{-1}= (1,1)$ to $(1,0)$ are just the endpoints. This is obvious.

Lemma~\ref{angle-at-conv} tells us that when we traverse the boundary of $R_0$, starting at $pt_1$ and travelling counterclockwise, we are always turning left at the corners. Thus $R_0$ is convex and we conclude that 
$R_0 = P_{a,n}^{(1)}.$
\end{proof}

\begin{lemma}\label{closest points}
\begin{enumerate}
\item If $O(m) < E(m)$, then the lattice point in the interior of $P_{a,n}$ that is closest to the side with endpoints $(0,0)$ and $(a,n)$ is $\mathbf{v}_{E(m)}$, and the lattice point in the interior of $P_{a,n}$ that is closest to the side with endpoints $(1,0)$ and $(a+1,n)$ is $\mathbf{u}_m-\mathbf{v}_{E(m)}$.
\item If $E(m) < O(m)$, then the lattice point in the interior of $P_{a,n}$ that is closest to the side with endpoints $(0,0)$ and $(a,n)$ is $\mathbf{u}_m -\mathbf{u}_{O(m)}$, and the lattice point in the interior of $P_{a,n}$ that is closest to the side with endpoints $(1,0)$ and $(a+1,n)$ is $\mathbf{u}_{O(m)}$.
\end{enumerate}
\end{lemma}

\subsection{Some numeric and geometric consequences} 

In the next lemma we describe some consequences of Theorem~\ref{main result}. For us, equation~\eqref{eq:interpretation} is the punchline of this article.

\begin{lemma}\label{vertex count and area}
    Let $[q_0,q_1,\ldots, q_m]$ denote the continued fraction of $n/a$. We have the following:

    \begin{equation}\label{eq:ver-count}
\#\textrm{vertices}(P_{a,n}^{(1)}) =
\begin{cases}
 2(m+1), & q_0 >1, \\ 2m, & q_0 =1.
 \end{cases}
    \end{equation}
    
\noindent The number of lattice points on the boundary of 
$P_{a,n}^{(1)}$ is $$\sum_{i=0}^m 2q_i -4,$$
and consequently, via Pick's theorem, we get that 
\begin{equation}\label{eq:area of hull}
\textrm{area}(P_{a,n}^{(1)})= n-\sum_{i=0}^m q_i. 
\end{equation}
Thus,
\begin{equation}\label{eq:interpretation}
    \sum_{i=0}^m q_i = \textrm{area}(P_{a,n}) - \textrm{area}(P_{a,n}^{(1)}).
\end{equation}

\end{lemma}

\begin{proof}
    In Theorem~\ref{main result} we listed the vertices of $P_{a,n}^{(1)}$. We obtain~\eqref{eq:ver-count} by observing that all of the lattice points in this list are distinct except when $q_0=1$. For the case $q_0=1$ we have  that $\mathbf{u}_{-1}=\mathbf{v}_0.$ 

    The only detail in this lemma that needs a little care is showing that the number of lattice points on the boundary of $P_{a,n}^{(1)}$ is $\sum_{i=0}^m 2q_i -4.$ 
    
    We consider the case when $q_0 >1$, and examine  Figure~\ref{fig:schematic picture}. We repeatedly invoke the fact that the number of lattice points on the line segment connecting $\mathbf{v}_{i-2}$ to $\mathbf{v}_i$ equals $q_i+1$ (see Lemma~\ref{no-int-pt}). By doing so we get the following.
    \begin{enumerate}
\item The number of lattice points on the line segment connecting $pt_7 = \mathbf{v}_m$ to $pt_6 =\mathbf{u}_m-\mathbf{u}_{O(m)}$ is $q_m+1$.
\item The number of lattice points on the line segment connecting $pt_8 = \mathbf{v}_0$ to $pt_1 =\mathbf{u}_{-1}$ is $q_0+1$.
    \item $$ \#(\gamma_2 \cap \Z^2) = \sum_{i=2, i \textrm{ even}}^{E(m)} q_i +1;$$
    \item $$ \#(\gamma_1 \cap \Z^2) = \sum_{i=1, i \textrm{ odd}}^{O(m)} q_i +1;$$
    \end{enumerate}
    
    We now consider the part of the boundary of $P_{a,n}^{(1)}$ that starts at $pt_6$ and ends at $pt_2$. Let us call this piece-wise linear path $\alpha$. From the above numerical information we get that the number of lattice points on $\alpha$ (and also on ${\mathcal S}(a)$)
is $\sum_{i=0}^m q_i -1$. Since 
$$\textrm{boundary}(P_{a,n}^{(1)}) = \alpha \cup {\mathcal S}(\alpha),$$ 
with $\alpha\cap {\mathcal S}(\alpha) = \{ pt_6,pt_2\}$, we conclude the number of lattice points on the boundary of $P_{a,n}^{(1)}$ equals $\sum_{i=0}^m 2q_i -4.$ The proof for the case $q_0=1$ is nearly identical.
\end{proof}

\noindent {\bf Remark.} When $a$ is small in comparison to $n$, then the dominant term in the sum $\sum_{i=0}^m q_i$ is $q_0=\lfloor n/a \rfloor$. Thus, when $a$ is small a good upper bound for area($P_{a,n}^{(1)}$) is $n-\lfloor n/a \rfloor.$

The length of the continued fraction of a rational number is simply  the number of steps taken by the Euclidean algorithm to find the greatest common divisor of the numerator and denominator of the rational number. Thus, standard results on the computational complexity of the Euclidean algorithm translate to results on the number of vertices of $P_{a,n}^{(1)}$.

Let us recall the relevant results. Our reference is  Volume 2, Seminumerical Algorithms, of Knuth's masterpiece {\it The Art of Computer Programming}~\cite{Kn}. (See Chapter 4, Section 5, Subsection 3.0.) We start by stating the 18th century theorem that describes the worst case for the running time of the Euclidean algorithm, or in other words,  an upper bound for the number of division steps. This occurs when the inputs are consecutive Fibonacci numbers.

\begin{theorem}\label{worst case}
    For $n\geq 1$, let $a$ and $b$ be integers with $a >b>0$ such that Euclid's algorithm 
    applied to $a$ and $b$ requires exactly $n$ division steps, and such that $a$ is as small as possible satisfying these conditions. Then $ a= F_{n+2}$ and $b=F_{n+1}$, where 
    $F_k$ is the $k$-th Fibonacci number.
\end{theorem}

An immediate consequence of ~\eqref{eq:ver-count} and Theorem~\ref{worst case} is that
\begin{equation}
    \# \textrm{ vertices of } P_{F_{n+1},F_{n+2}}^{(1)} = 2n;
\end{equation}
thus demonstrating that (unlike lattice rectangles) there is no universal bound for the number of vertices of $P_{a,n}^{(1)}$. 

The Fibonacci numbers provide another neat result about $P_{a,n}^{(1)}.$

\begin{cor}
$P_{a,n}^{(1)}$ is a clean lattice polygon if and only if
 $n=F_m$, and $a=F_{m-1}$ or $a=F_{m-2}$.
\end{cor}

\begin{proof}
    $P_{a,n}^{(1)}$ is clean if and only if the continued fraction of $n/a$ is $[1,1,1,\ldots,1,2]$ or 
    $[2,1,1,\ldots, 1,2]$. If it equals $[1,1,1,\ldots,1,2]$,
    then $n=F_m,a=F_{m-1}$; if it equals $[2,1,1,\ldots,1,2]$,
    then $n=F_m,a=F_{m-2}$.
\end{proof}

A consequence of Theorem~\ref{worst case} is that the number of division steps needed to find $\gcd(a,n)$ using the Euclidean algorithm is at most $2.078\log(n) + .6723$, and consequently  
\begin{equation}
    \# \textrm{ vertices of } P_{a,n}^{(1)} \leq 4.156 \log(n)+ 
    3.3446.
\end{equation}

\subsubsection{A couple of averages}
We now discuss the average value for the number of vertices of 
$P_{a,n}^{(1)}$ for fixed $n$. Let $T(m,n)$ denote the number of division steps needed when the inputs to the Euclidean algorithm are $m$ and $n$; and let 
$$\tau_n= \frac{1}{\varphi(n)} \sum_{0\leq m <n, \gcd(m,n)=1} T(m,n).$$
It has been shown that 
\begin{equation}
    \tau_n \approx\frac{12\log (2)}{\pi^2} \log(n) + 1.467.
\end{equation}
(See~\cite[page 372, equation 60]{Kn}). From this it follows that 
the average number of vertices of $P_{a,n}^{(1)}$ is approximately 
$1.685 \log(n).$

Furthermore, for a fixed $n$, 
the average value of $\textrm{area}(P_{a,n}^{(1)})$ is approximately 
$n- (6/\pi^2)\log^2 n.$ We obtain this by applying a result of  
Popov~\cite{P}. Let us explain.

For $a,n \in \Z^+$, with $\gcd(a,n)=1$, let $S(a,n)$ denote the sum of the partial quotients of the continued fraction of 
$a/n$. For $p$ prime, Popov proved that
\begin{equation}
\sum_{1 \leq a < p} S(a,p) = \frac{6}{\pi^2} p \log^2(p) + O(p\log(p)).
\end{equation}
Furthermore, he stated, without proof, that for large $n$
\begin{equation}
    \sum_{1 \leq a <n} S(a,n) \approx \frac{6}{\pi^2} \varphi(n) \log^2(n) .
\end{equation}

If the continued fraction of $n/a$, with $n >a>0$, is
$[q_0,q_1,\ldots, q_m]$, then the continued of 
$a/n$ is $[0,q_0,q_1,\ldots, q_m]$. Consequently,
$S(n,a)=S(a,n).$
Equation~\eqref{eq:area of hull} is 
$$\textrm{area}(P_{a,n}^{(1)}) = n- S(a,n).$$
Consequently, we can invoke Popov's result to conclude that on average $\textrm{area}(P_{a,n}^{(1)})$ is approximately 
$n-(6/\pi^2) \log^2 n$.\footnote{As an aside, we mention that Viktor Nikolaevich Popov was an eminent theoretical physicist highly regarded for his work on the quantization of non-abelian gauge fields.}

\subsubsection{Speculative connections to modular forms?} We now indulge in a short interlude of fanciful speculation! From the analytic number theory point of view, the average number of vertices of $P_{a,n}^{(1)}$, with fixed $n$, is reciprocal to the density of prime numbers less than $n$. We wonder if this can be viewed as a bridge that connects the levels of interior hulls of lattice points and the levels of modular forms (or more generally automorphic forms) and the modular form spaces dimensions for each level, obtained from Riemann-Roch, where the level of the interior hull corresponds to the level of congruence subgroup and the lattice points count correspond to the level $N$ modular form space's dimension. 

To elaborate, our ``reasoning'' is as follows.  Heuristically speaking, the number of lattice points in the interior of a polygon has the same magnitude as the area of the polygon. In the case of modular forms, the volume is linked to the spectrum of certain invariant operators  via trace formulae. For example, consider a family of Hecke operators acting on a vector space of modular forms of a given weight and a given level of congruence subgroup of the modular group. The trace of the identity operator gives the dimension of that space. That is, it gives the dimension of the space of modular forms with given weight and level. This quantity is traditionally calculated via Riemann–Roch, or trace formulae such as Selberg’s or Arthur’s. In particular, a complete set of dimension formulae have been computed in~\cite[Proposition 6.1]{St}. To us, these dimension formulae have a similar flavour to our main counting formula ~\eqref{eq:interpretation}. 

\subsection{Exercise for the curious reader.} An arbitrary lattice parallelogram, $P$, can be tiled by copies of a clean parallelogram, $C$. It is left to the curious reader to work out how we can determine $P^{(1)}$ from knowing $C^{(1)}$.

\subsection{Some future directions.} We have two future lines of inquiry. 
\begin{enumerate}
    \item The first is to investigate the interior hulls $P^{(2)}, P^{(3)}, \ldots$ for clean lattice parallelograms to determine if there are other interesting correspondences. All of our efforts in this paper have been confined to the first interior hull $P^{(1)}$.
\item The second is to investigate the interior hulls of clean lattice parallelepipeds. In the next subsection we describe a particular case for such parallelepipeds.
\end{enumerate}

\subsection{An interesting degenerate case in $\R^3$.} White~\cite{Wh} proved that all of the non-vertex lattice points in  a clean parallelepiped that arises from an empty tetrahedra are \emph{coplanar}. Consequently the interior hull of such a parallelepiped is always a lattice polygon and \emph{never} a lattice polytope. It follows that our central result applies \emph{mutatis mutandis} in this special case. Let us elaborate.

Let $\P$ be a clean lattice parallelepiped. Without loss of generality we can assume that one of the vertices is at the origin and consequently we can write $\P$ as 
$$ \P= \{ t_1 \mathbf{v}_1+ t_2\mathbf{v}_2 + t_3 \mathbf{v}_3 \, ; \; 
\mathbf{v}_1, \mathbf{v}_2, \mathbf{v}_3 \in \Z^3, 0 \leq t_1,t_2,t_3 \leq 1 \}.$$

Since $\P$ is clean we can find an unimodular map that sends 
$$ \mathbf{v}_1 \mapsto (1,0,0), \mathbf{v}_2 \mapsto (0,1,0), \mathbf{v}_3 \mapsto (a,b,c),$$
with $a,b,c \in \Z, c = vol(\P), 1 \leq a,b <c, \gcd(a,c)=\gcd(b,c)=1.$ Thus we can identify $\P$ with the parallelepiped
$$ \P_{a,b,c}= \{ t_1 (1,0,0)+ t_2(0,1,0) + t_3(a,b,c) \, ; \; 
0 \leq t_1,t_2,t_3 \leq 1 \}. $$
Let ${\mathcal T}_{a,b,c}$ denote the tetrahedra, 
$$ \{ t_1 (1,0,0) + t_2 (0,1,0) + t_3 (a,b,c) \, ; \; 0 \leq t_1,t_2,t_3 \leq 1, 
 0\leq t_1+t_2+t_3 \leq 1\}.$$
We now come to the punchline. If ${\mathcal T}_{a,b,c}$ is empty (\emph{i.e.} does not contain any lattice points other than its vertices), then $\P_{a,b,c}^{(1)}$ can be identified with 
 $P_{x,c}^{(1)}$ for some $x \in \Z$ with $1 \leq x < c, \gcd(x,c)=1.$ The key is to invoke White's theorem~\cite{Wh}.  White showed that if ${\mathcal T}_{a,b,c}$ is empty then either $a=1$ or $b=1$ or $d=1$, where $d=(1-a-b) \mod c.$ From this it follows that if ${\mathcal T}_{a,b,c}$ is empty then $\P_{a,b,c}$ is unimodularly equivalent to $\P_{1,x,c}$ for some $x<c$ with $\gcd(x,c)=1$. All of the interior lattice points in $\P_{1,x,c}$ lie in the parallelogram with vertices $(1,0,0),(1,1,0), (1,x,c)$ and $(1,x+1,c)$. Clearly, we can identify this parallelogram with the $P_{x,c} \subset \R^2,$ and consequently $\P_{a,b,c}^{(1)}$ can be identified with $P_{x,c}^{(1)} \subseteq \R^2$.

The interested reader can consult the book~\cite[Chapter 15]{K} as well as the expository article~\cite{KR} for further information on White's theorem.

\begin{figure}
    \centering
    \includegraphics[width=1\linewidth]{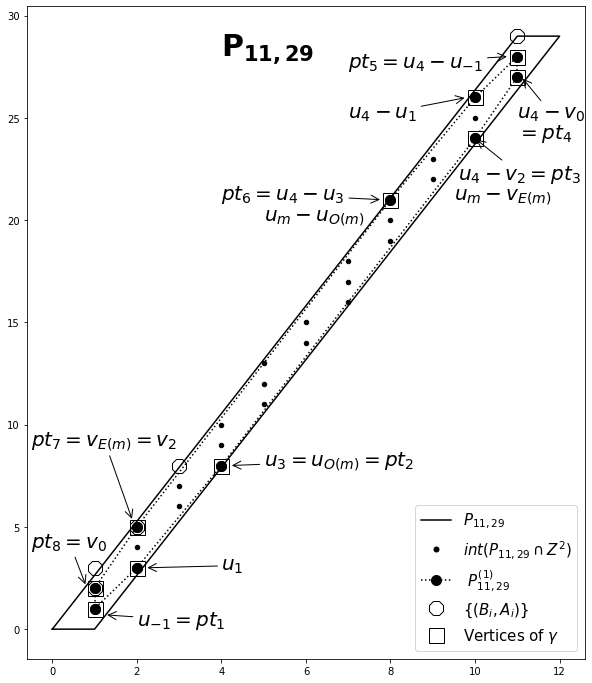}
    \caption{$P_{11,29}$ and $P_{11,29}^{(1)}$}
    \label{fig:P11,29}
\end{figure}

\section{Appendix: Graphing Routines}

We give the Python code to draw $P_{a,n}$, $P_{a,n}^{(1)}$, the lattice points in $P_{a,n}$, the points corresponding to the convergents of $n/a$ and the points defining the boundary of $P_{a,n}^{(1)}$.

The code has been used in drawing Figure~$\ref{fig:P11,29}$. In that figure we have the following: $P_{11,29}$, $\textrm{int}(P_{11,29} \cap \Z^2)$, 
$ P_{11,29}^{(1)} $, the points corresponding to the convergents of $29/11$, and the points defining the paths arising in Theorem~\ref{main result}.

\begin{lstlisting}[language=Python]
import numpy as np
import matplotlib.pyplot as plt
import math
import matplotlib.lines as mlines
from scipy.spatial import ConvexHull

# Draw the parallelogram P 
# with vertices (0,0), (a,n), (a+1, n), (1,0)
# where a, n are integers and 0<a<n 
def plotparll(a,n):
    parl = np.array(
        [(0,0), (a,n), (a+1, n), (1,0), (0,0)])
    x, y = parl.T
    plt.plot(x, y, 'k')

# List the interior points of parallelogram P
def interior (a,n):
    list2=[]
    for i in range(1, a+1):
        for j in range(1, n):
            if (j>=(n/a)*(i-1)) \
            and (j<= (n/a)*i):
                list2=list2+[(i,j)]
    return(list2)


# Draw the interior points of the parallelogram P
def drawintparl(a, n):
    A = np.array(interior(a, n))
    u, v = A.T
    plt.scatter(u, v, c='k', s=20)
    
# Calculate the convergents for n/a 
def cf_conv(a,n):
    vm2 = np.array([1,0])
    vm1 = np.array([0,1])
    q = 1; i = 0
    v = np.add(vm2, q*vm1)
    w = np.add(vm2, (q+1)*vm1)
    pts = []
    while (v[1]*a != v[0]*n):
        if (i%2 == 0):
            while((v[1]*a < v[0]*n) 
                  and (w[1]*a < w[0]*n)):
                q=q+1
                v= np.add(vm2, q*vm1)
                w= np.add(vm2, (q+1)*vm1)
        elif (i%2 == 1):
            while((v[1]*a > v[0]*n) 
                  and (w[1]*a > w[0]*n)):
                q=q+1
                v= np.add(vm2, q*vm1)
                w= np.add(vm2, (q+1)*vm1)
        if ((v[1]*a == v[0]*n) 
            or (w[1]*a == w[0]*n)):
            pts = pts + [w]
            ptsarr= np.asarray(pts)
            return(ptsarr)
            break
        i = i+1
        vm2 = vm1; vm1 = v
        pts = pts + [v]
        q = 1
        v = np.add(vm2, q*vm1) 
        w = np.add(vm2, (q+1)*vm1)            

# Plot the points corresponding to the 
# convergents for n/a
def cf_conv_plot(a,n):
    u6, v6 = cf_conv(a,n).T
    plt.plot(
        u6, v6, '8',mec='k', 
        c='none', ms=15)
        
# Draw the convex hull of the 
# interior lattice points of P
def drawcvhull(a, n):
    if a == 1:
        plt.plot((1,1),(1,n-1),'o:k',ms=10)
    elif a == n-1:
        plt.plot((1,n-1),(1,n-1),'o:k',ms=10)
    else:
        points = np.array(interior(a, n))
        hull = ConvexHull(points)
        for simplex in hull.simplices:
            plt.plot(
                points[simplex, 0], 
                points[simplex, 1],'o:k',ms=10)

# Calculate the points in the curve gamma 
# computed from the convergents of n/a
# that gives the convex hull of the 
# interior points of P(a,n)
def gammapts(a,n):
    # Calculate the convergents of n/a 
    #i.e. v(0), v(1), ..., v(m).
    V = cf_conv(a,n)
    
    # Calculate m
    m = len(V)-1
    
    # Calculate o_m = smallest odd integer < m
    # and e_m = smallest even integer < m
    if m%2 == 0:
        o_m = m-1
        e_m = m-2
    else:
        o_m = m-2
        e_m = m-1   
        
    # Calculate u(i) = v(i) + (1,0) 
    # for i = 0 to m. Note u(-1)= (1,1)
    # List u(-1), u(1),..., u(o_m) (List 1). 
    # Initialise List1=[(1,1)]
    List1 = [np.array([1, 1])]  
    for i_o in range(0,o_m+1):
        if i_o%2 == 1:
            List1.append(V[i_o]+(1,0))
            
    # Calculate u(m)-v(e_m), u(m)-v(e_m-2),.., 
    # u(m)-v(0) (List 2)
    List2 = []
    for i_e in range(0,e_m+1):
        if i_e%2 == 0:
            List2.append(V[m]+(1,0)
                         -V[(e_m - i_e)])
    
    # Calculate u(m)-u(-1), u(m)-u(1),... 
    # u(m)-u(o_m) (List 3)
    # Initialise List3 to u(m)-u(-1)
    List3 = [V[m]+(1,0)- (1,1)]
    for i_o in range(0,o_m+1):
        if i_o%2 == 1:
            List3.append(V[m]+(1,0)
                         -(V[i_o]+(1,0)))
            
    # List ve_m, v(e_m-2), ... , v0 (List 4)
    List4 = []
    for i_e in range(0,e_m+1):
        if i_e%2 == 0:
            List4.append(V[(e_m - i_e)])
    
    # Concatenate List1, List2, List3, List4
    return(List1 + List2 + List3 + List4)

# Plot the points that define gamma
def gammaplot(a,n):
    u7, v7 = np.array(gammapts(a,n)).T
    plt.plot(
        u7, v7,'s',mec='k',c='none',ms=15)

# Check that 0<a<n and a,n are integers
# Draw P(a,n), lattice pts inside P,
# convex hull of interior pts, 
# convergents of n/a, pts defining 
# the curve gamma
def drawall(a,n):
    if (isinstance(a, int) 
        and isinstance(n, int) and 0<a<n):
        plotparll(a,n)
        drawintparl(a,n)
        cf_conv_plot(a,n)
        drawcvhull(a,n)
        gammaplot(a,n)
    else:
        print('a, n must be integers with 0<a<n')
        
\end{lstlisting}

\




\end{document}